\documentclass[11pt,a4paper]{amsart}
\usepackage[margin=25mm]{geometry}
\usepackage[numbers]{natbib}
\usepackage{hyperref}

\usepackage{amssymb,amsmath}
\usepackage{amsthm}
\usepackage{bbm, bm}
\usepackage{stmaryrd}
\usepackage[usenames,dvipsnames]{xcolor}
\usepackage{color}
\usepackage{graphicx}
\usepackage{caption}
\usepackage{float}
\usepackage{wrapfig}

\usepackage{lineno}

\usepackage{abstract}

\input xy
\xyoption{all} 

\numberwithin{equation}{section}

\newtheorem{theorem}{Theorem}[section]
\newtheorem{corollary}[theorem]{Corollary}

\newtheorem{proposition}[theorem]{Proposition}

\theoremstyle{definition}
\newtheorem{definition}[theorem]{Definition}
\newtheorem{example}[theorem]{Example}

\newenvironment{remark}[1][Remark.]{\begin{trivlist}
\item[\hskip \labelsep {\bfseries #1}]  }{ \end{trivlist}}

\newcommand{\InSec}{\mbox{$\underline{\Sec}$}}

\DeclareMathOperator{\End}{End}
\DeclareMathOperator{\Vect}{Vect}

\font\black=cmbx10 \font\sblack=cmbx7 \font\ssblack=cmbx5 \font\blackital=cmmib10  \skewchar\blackital='177
\font\sblackital=cmmib7 \skewchar\sblackital='177 \font\ssblackital=cmmib5 \skewchar\ssblackital='177
\font\sanss=cmss10 \font\ssanss=cmss8 
\font\sssanss=cmss8 scaled 600 \font\blackboard=msbm10 \font\sblackboard=msbm7 \font\ssblackboard=msbm5
\font\caligr=eusm10 \font\scaligr=eusm7 \font\sscaligr=eusm5  \font\fraktur=eufm10
\font\sfraktur=eufm7 \font\ssfraktur=eufm5 
\font\bsymb=cmsy10 scaled\magstep2
\def\all#1{\setbox0=\hbox{\lower1.5pt\hbox{\bsymb
       \char"38}}\setbox1=\hbox{$_{#1}$} \box0\lower2pt\box1\;}
\def\exi#1{\setbox0=\hbox{\lower1.5pt\hbox{\bsymb \char"39}}
       \setbox1=\hbox{$_{#1}$} \box0\lower2pt\box1\;}

\def\tx#1{{\fam0\relax#1}}

\newfam\bifam
\textfont\bifam=\blackital \scriptfont\bifam=\sblackital \scriptscriptfont\bifam=\ssblackital

\newfam\blfam
\textfont\blfam=\black \scriptfont\blfam=\sblack \scriptscriptfont\blfam=\ssblack

\newfam\bbfam
\textfont\bbfam=\blackboard \scriptfont\bbfam=\sblackboard \scriptscriptfont\bbfam=\ssblackboard

\newfam\ssfam
\textfont\ssfam=\sanss \scriptfont\ssfam=\ssanss \scriptscriptfont\ssfam=\sssanss
\def\sss#1{{\fam\ssfam\relax#1}}

\newfam\clfam
\textfont\clfam=\caligr \scriptfont\clfam=\scaligr \scriptscriptfont\clfam=\sscaligr

\newfam\frfam
\textfont\frfam=\fraktur \scriptfont\frfam=\sfraktur \scriptscriptfont\frfam=\ssfraktur

\def\hpb#1{\setbox0=\hbox{${#1}$}
    \copy0 \kern-\wd0 \kern.2pt \box0}
\def\vpb#1{\setbox0=\hbox{${#1}$}
    \copy0 \kern-\wd0 \raise.08pt \box0}

\def\pmb#1{\setbox0\hbox{${#1}$} \copy0 \kern-\wd0 \kern.2pt \box0}
\def\pmbb#1{\setbox0\hbox{${#1}$} \copy0 \kern-\wd0
      \kern.2pt \copy0 \kern-\wd0 \kern.2pt \box0}
\def\pmbbb#1{\setbox0\hbox{${#1}$} \copy0 \kern-\wd0
      \kern.2pt \copy0 \kern-\wd0 \kern.2pt
    \copy0 \kern-\wd0 \kern.2pt \box0}
\def\pmxb#1{\setbox0\hbox{${#1}$} \copy0 \kern-\wd0
      \kern.2pt \copy0 \kern-\wd0 \kern.2pt
      \copy0 \kern-\wd0 \kern.2pt \copy0 \kern-\wd0 \kern.2pt \box0}
\def\pmxbb#1{\setbox0\hbox{${#1}$} \copy0 \kern-\wd0 \kern.2pt
      \copy0 \kern-\wd0 \kern.2pt
      \copy0 \kern-\wd0 \kern.2pt \copy0 \kern-\wd0 \kern.2pt
      \copy0 \kern-\wd0 \kern.2pt \box0}

\mathchardef\za="710B  
\mathchardef\zb="710C  
\mathchardef\zg="710D  
\mathchardef\zd="710E  
\mathchardef\zve="710F 
\mathchardef\zz="7110  
\mathchardef\zh="7111  
\mathchardef\zvy="7112 
\mathchardef\zi="7113  
\mathchardef\zk="7114  
\mathchardef\zl="7115  
\mathchardef\zm="7116  
\mathchardef\zn="7117  
\mathchardef\zx="7118  
\mathchardef\zp="7119  
\mathchardef\zr="711A  
\mathchardef\zs="711B  
\mathchardef\zt="711C  
\mathchardef\zu="711D  
\mathchardef\zvf="711E 
\mathchardef\zq="711F  
\mathchardef\zc="7120  
\mathchardef\zw="7121  
\mathchardef\ze="7122  
\mathchardef\zy="7123  
\mathchardef\zf="7124  
\mathchardef\zvr="7125 
\mathchardef\zvs="7126 
\mathchardef\zf="7127  
\mathchardef\zG="7000  
\mathchardef\zD="7001  
\mathchardef\zY="7002  
\mathchardef\zL="7003  
\mathchardef\zX="7004  
\mathchardef\zP="7005  
\mathchardef\zS="7006  
\mathchardef\zU="7007  
\mathchardef\zF="7008  
\mathchardef\zW="700A  
\mathchardef\zC="7009  

\newcommand{\be}{\begin{equation}}
\newcommand{\ee}{\end{equation}}

\newcommand{\bea}{\begin{eqnarray}}
\newcommand{\eea}{\end{eqnarray}}
\def\*{{\textstyle *}}
\newcommand{\R}{{\mathbb R}}

\newcommand{\Z}{{\mathbb Z}}

\newcommand{\s}{{\textstyle *}}

\def\Sec{\sss{Sec}}
\def\Vect{\sss{Vect}}

\def\sT{{\sss T}}

\def\xi{\tx{i}}

\def\s*{{\scriptstyle *}}

\def\cC{\mathcal{C}}

\newcommand{\beas}{\begin{eqnarray*}}
\newcommand{\eeas}{\end{eqnarray*}}

\title{The Ternary Structure of Lie Algebroid Connections} 
\author{Andrew James Bruce} 
   \address{Department of Mathematics,
The Computational Foundry,
Swansea University Bay Campus,
Fabian Way,
Swansea, SA1 8EN}  
   \email{andrewjamesbruce@googlemail.com}

   \date{\today}
 
\begin{document}
 \maketitle
\vspace{-25pt}
\begin{abstract}{\noindent }\\
We examine the abelian heap of linear connections on anchored vector bundles and Lie algebroids. We show how the ternary structure on the set of linear connections `interacts' with the torsion and curvature tensors. The endomorphism truss of linear connections is constructed.\\
\noindent {\Small \textbf{Keywords:} Heaps;~Trusses;~Linear Connections;~Lie Algebroids}
\end{abstract}
\smallskip 

\noindent \emph{Dedicated to the memory of Sandra Irene Bruce}

\section{Introduction and Preliminaries}
\subsection{Introduction} We need hardly mention that the notion of a connection in its various forms is of vital importance in differential geometry and geometric approaches to physics (see, for example, \cite{Mangiarotti:2000}).   As an important example of the r\^{o}le of connections in modern mathematics, we point to the construction of characteristic classes of principal bundles via  Chern--Weil theory. In physics, connections are related to gauge fields and are vital in general relativity and other geometric approaches to gravity such as metric-affine gravity. Connections are also found in geometric approaches to relativistic mechanics. In short, connections are found throughout modern geometry and physics.   \par 
In this note, we study linear connections on anchored vector bundles, and especially Lie algebroids in the category of smooth real supermanifolds (in the sense of Berezin \& Le\u{\i}tes, see \cite{Berezin:1975}). Such connections, here formulated following Koszul as covariant derivatives, generalise affine connections on manifolds, and so our results restrict directly to the classical setting.  Lie algebroids are a `mix' of tangent bundles and Lie algebras and offer a wide geometric framework to formulate classical geometric notions (see for example Mackenzie \cite{Mackenzie:2005}).\par
Heaps were first defined and studied by Pr\"{u}fer \cite{Prufer:1924} and  Baer \cite{Baer:1929} as a set equipped with a ternary operation satisfying simple axioms, including a ternary generalisation of associativity. A heap can, loosely, be thought of as a group in which the identity had been discarded. Given some group, the ternary operation $(a,b,c) \mapsto ab^{-1}c$ defines a heap.  For example, in an affine space, one can construct a heap operation as $(u,v,w) \mapsto u - v+w$.  In the other direction, by selecting any element in a heap, one can reduce the ternary operation to a  group operation, such that the chosen element is the identity element. Our main reference for heaps and related structures is the book by Hollings,  \& Lawson \cite{Hollings:2017}, which presents translations of Wagner's original works. \par 
Via the above paragraph, it is clear that, as they form an affine space, the set of affine/linear connections on a (super)manifold/vector bundle form a heap. We investigate the consequences of this heap structure.  One interesting result is the fact that the set of endomorphisms of connections on an anchored vector bundle (and so on a Lie algebroid) comes with the structure of a \emph{truss}. The latter structures are ring-like algebraic structures in which the binary addition is replaced with a heap operation together with some natural distributivity axioms (see \cite{Brzezinski:2018,Brzezinski:2019,Brzezinski:2020,Brzezinski:2019b,Brzezinski:2022}). In part, the output of this note is the construction of a geometric example of a truss.    \par 
We remark that although we are working in supergeometry, none of the results of this note hinges on that fact. We assume the reader has some familiarity with supermanifolds and supervector bundles. However, we will work in a coordinate-free way with globally defined objects. For details of supermanifolds, the reader may consult Carmeli, Caston L. \& Fioresi \cite{Carmeli:2011}, for example. For details of (super) vector bundles, one may consult \cite{Balduzzi:2011}. Local expressions for Lie algebroids can be found in the introductory section of \cite{Bruce:2021}.  We will often neglect the prefix `super' and unless explicitly stated, everything will be $\Z_2$-graded. \par 
\subsection{Algebraic Preliminaries}\label{subsec:AlgPre} A \emph{semiheap} $H$ is a set (possibly empty) equipped with a ternary operation $(a,b,c) \mapsto [a,b,c]$ that is para-associative, i.e.,
$$[[a, b, c], d, e] = [a, [d, c, b], e] = [a, b, [c, d,e]]\, .$$
 A semiheap is said to be \emph{abelian} is $[a,b,c] = [c,b,a]$ for all $a,b$ and $c \in H$. If all the elements are \emph{bi-unitary}, that is, $[a,b,b] = a$ and $[b,b,a]=a$ for all $a$ and $b \in H$, then we have a \emph{heap}. We recall that a \emph{left truss} is an abelian heap together with an associative binary operation that distributes over the ternary operation (from the left), i.e., 
$$a \cdot [b,c,d] = [a \cdot b, a\cdot c , a\cdot d]\,. $$
Similarly, a \emph{right truss} can be defined. If we have both left and right distributivity, then we speak of a \emph{truss}. We divert the reader to the original literature on trusses (see \cite{Brzezinski:2018,Brzezinski:2019,Brzezinski:2020,Brzezinski:2019b,Brzezinski:2022}). A video lecture outlining the theory of trusses, including motivation and a historical perspective is \cite{NorthAtlanticNCGSeminar:2021}.

\section{Heaps of Connections}
\subsection{Connections on Anchored Vector Bundles}
Let us recall the definition of an anchored vector bundle, here adapted to the setting of supergeometry.
\begin{definition}
A vector bundle (in the category of supermanifolds) $\pi : A \rightarrow M$ is said to be an \emph{anchored vector bundle} if it is equipped with a vector bundle homomorphism (over the identity) $\rho: A \rightarrow \sT M$, which is referred to as the \emph{anchor}.
\end{definition}
The space of sections of a vector bundle is $\Z_2$-graded, i.e., $\InSec(A) =  \Sec_0(A)\oplus_{C^\infty(M)} \Sec_1(A)$. By minor abuse of notation, we will also  write
\begin{equation}
\rho :  \InSec(A) \longrightarrow \Vect(M)\,,
\end{equation}
for the associated (even) homomorphism of $C^\infty(M)$-modules. We will, when convenient write $\rho_u := \rho(u)$, where $u \in\InSec(A)$. The Grassmann parity (or degree) of sections (as well as functions, tensors etc.) we denote using `tilde', i.e., $\widetilde{u} \in \Z_2$.
\begin{example}
Let $\tau : E \rightarrow  M$ be a vector bundle over the supermanifold $M$. Then, as fibre products exist in the category of supermanifolds, $A := \sT M \times_M E$ is also a vector bundle over $M$. The anchor is the projection onto the first factor (a little care is needed as we have a locally ringed space, but this can be made sense of using local coordinates, for example).
Sections are clearly $\InSec(A)=  \Vect(M)\bigotimes_{C^\infty(M)} \InSec(E)$. Then, on pure tensor products, $\rho(X\otimes u) =  X$.
\end{example}
\begin{definition}\label{Def:LinCon}
A \emph{linear connection} on an anchored vector bundle $(A, \rho)$ is an $\R$-bilinear map
$$\nabla :  \InSec(A) \times \InSec(A) \longrightarrow \InSec(A)\,,$$
such that 
\begin{enumerate}
\setlength\itemsep{0.5em}
\item $\widetilde{\nabla_u v} =  \widetilde{u} + \widetilde{v}$,
\item $\nabla_{f u}v = f \, \nabla_u v$, and
\item $\nabla_u (f v) =  \rho_u(f)v + (-1)^{\widetilde{u}\, \widetilde{f}}\, f \nabla_u v$,
\end{enumerate}
\smallskip 

for all $f \in C^\infty(M)$  and $u,v \in \InSec(A)$.
\end{definition}
We will denote the set (or affine space) of linear connections on $(A, \rho)$ as $\cC(A)$.
\begin{remark}\
\begin{enumerate}
\setlength\itemsep{0.5em}
\item In this note we only consider Grassmann even connections. Odd connections are not a truly separate notion as uncovered in  \cite{Bruce:2020}.
\item The existence of Lie algebroid connections, and so anchored vector bundles is established in  \cite{Krizka:2008}, for example, using a partition of unity. For real supermanifolds (as locally ringed spaces) we always have partitions of unity.
\item There is the related notion of an $A$-valued connection $\hat{\nabla}: \InSec(A)\times \InSec(E) \rightarrow \InSec(E)$. As we will want to discuss torsion, linear connections as defined above are needed. 
\item Linear connections can be reformulated as odd vector fields on a particular bi-graded supermanifold build from the initial anchored vector bundle, see \cite{Bruce:2019} for details. We will avoid graded/weighted geometry in this note and stick to a more classical presentation. 
\end{enumerate}
\end{remark}
\begin{example}
A real vector space $V$, non-super for simplicity, can be considered as an anchored vector bundle over a single point, where the anchor is the zero map.  A linear connection on $V$ is simply an $\R$-bilinear map $\nabla: V \times V \rightarrow V$.
\end{example}
\begin{example}
Let $\pi : A \rightarrow M$ be an arbitrary vector bundle (in the category of supermanifolds). This vector bundle can be considered as an anchored vector bundle by setting the anchor to be the zero map.  As the zero vector field can be considered as both even and odd, this choice is consistent. We refer to such structures as zero-anchored vector bundles. Then a linear connection on a zero vector bundle is an even $\R$-bilinear map such that 
$$\nabla_{fu}v = (-1)^{\widetilde{u} \widetilde{f}}\, \nabla_u (f v) = f \nabla_u v\,.$$
That is, a linear connection in this context is a bilinear form on the $C^\infty(M)$-module $\InSec(A)$.
\end{example}
\begin{example}
Given a vector bundle connection $\bar{\nabla}: \Vect(M)\times \InSec(A) \rightarrow \InSec(A)$ on $(A, \rho)$, one has a canonically associated linear connection by setting $\nabla_u v := \bar{\nabla}_{\rho(u)} v$.
\end{example}
We define a ternary operation on the set of linear connections on an anchored bundle as
 \begin{equation}\label{eqn:TerOP}
 [\nabla^{(1)}, \nabla^{(2)},  \nabla^{(3)}] := \nabla^{(1)} - \nabla^{(2)}+ \nabla^{(3)}\,,
 \end{equation}
 for arbitrary  $ \nabla^{(i)} \in \cC(A)$. This ternary product does indeed produce another linear connection. Note, of course, that the sum or difference of two linear connections is not a linear connection. The reader can easily verify the following proposition (see Subsection \ref{subsec:AlgPre}).
\begin{proposition}
Let $(A, \rho)$ be an anchored vector bundle. Then the set of linear connections, $\cC(A)$, is an abelian heap with the ternary operation being defined by
\eqref{eqn:TerOP}.
\end{proposition}
From the general theory of heaps, we know that if we fix some connection $\nabla^{(0)} \in \cC(A)$, then we have an associated abelian group structure on $\cC(A)$ given by $\nabla^{(1)}\bullet_{\nabla^{(0)}} \nabla^{(2)} := [\nabla^{(1)}, \nabla^{(0)}, \nabla^{(2)}]$, and the inverse operation is given by $(\nabla)^{-1} := [\nabla^{(0)}, \nabla, \nabla^{(0)}]$.  All such abelian groups associated with a different choice of reference connection are isomorphic.
\begin{example}\label{exp:RieMan}
Let $(M, g)$ be a Riemannian manifold. Then we have the canonical Levi-Civita connection on the tangent bundle $\sT M$, which we denote as $\nabla^{0}$. We consider the tangent bundle as an anchored vector bundle where the anchor is the identity map. Thus the set of affine connections on $(M,g)$ is canonically an abelian group with the group product and inverse being
\begin{align*}
& \nabla^{(1)}\cdot \nabla^{(2)} := [\nabla^{(1)}, \nabla^{0}, \nabla^{(2)}]\,,\\
& (\nabla)^{-1} := [\nabla^{0}, \nabla, \nabla^{0}]\,.
\end{align*}
If $\nabla^{(i)}$ are metric connections, i.e., $\nabla g =0$, then, due to linearity of the metric, $[\nabla^{(1)}, \nabla^{(2)}, \nabla^{(3)}]$ is also a metric connection, that is, metric connections form a subheap of the heap of all affine connections. In turn, we also have a subgroup of metric connections. Moreover, given an arbitrary affine connection, $\nabla$, we note that 
$$\nabla = [\nabla, \nabla^0 , \nabla^0] =   \nabla^0 + (\nabla- \nabla^0)\,,$$
which gives a ``heapy'' origin to the well-known fact that any affine connection on a Riemannian manifold is the Levi-Civita connection plus a tensor of type $(1,2)$.
\end{example}
\begin{remark}
The previous example directly generalises to even and odd Riemannian supermanifolds as we again have a canonical  Levi-Civita connection.
\end{remark}
A section $u \in \InSec(A)$ is said to be \emph{auto-parallel} if there exists a linear connection $\nabla \in \cC(A)$ such that $\nabla_u u = 0$. The set of all such linear connections we denote as $\cC(A,u)$. The following proposition is evident.
\begin{proposition}
Let $(A, \rho)$ be an anchored vector bundle. The set $\cC(A,u)$ for any section $u \in \InSec(A)$, is closed under the ternary operation \eqref{eqn:TerOP}.
\end{proposition}
\subsection{Torsion and Curvature}
To discuss torsion and curvature in the setting of anchored vector bundles we require a bracket on the space of sections. We focus on the situation of Lie algebroids, though once can relax the Jacobi identity and the compatibility of the anchor and the bracket if desired.  For completeness, we define a Lie algebroid following  Pradines (see \cite{Pradines:1974}), modified to the setting of supergeometry. 
\begin{definition}
An anchored vector bundle $(A, \rho)$ is a \emph{Lie algebroid} if the space of sections $\InSec(A)$ comes equipped with an $\R$-bilinear map - the Lie bracket
$$[-,-] :  \InSec(A)\times \InSec(A) \longrightarrow \InSec(A)\,,$$
that satisfies the following:
\begin{enumerate}
\setlength\itemsep{0.5em}
\item $\widetilde{[u,v]} = \widetilde{u} + \widetilde{v}$,
\item $[u,v] = - (-1)^{\widetilde{u} \widetilde{v}} \, [v,u]$,
\item $[u,f v] = \rho_u(f)v +(-1)^{\widetilde{u} \widetilde{f}}\, f [u,v]$,
\item $[u,[v,w]] = [[u,v],w] + (-1)^{\widetilde{u}\widetilde{v}}\, [v,[u,w]]$,
\end{enumerate}
\smallskip

for all $u,v$ and $w \in \InSec(A)$ and $f \in C^\infty(M)$.
\end{definition}
The above conditions imply that the anchor is a homomorphism of Lie algebras, i.e.,
 $$\rho_{[u,v]} = [\rho(u), \rho(v)]\,.$$
\begin{example}
The tangent bundle of a supermanifold $\sT M$ is a Lie algebroid with the anchor being the identity map and the Lie bracket being the standard commutator of vector fields. 
\end{example}
\begin{example}
Let $\mathfrak{g} = \mathfrak{g}_0 \oplus \mathfrak{g}_1$ be a super vector space. Then associated with this via the `manifoldcation' functor, is the linear supermanifold, here thought of as a vector bundle over a point $\mathfrak{g}^{\textnormal{man}} \rightarrow \star$. The important aspect of the construction is that $\mathfrak{g} \cong \InSec(\mathfrak{g}^{\textnormal{man}})$. If $\mathfrak{g}$ is a Lie algebra, then $\mathfrak{g}^{\textnormal{man}}$ is a Lie algebroid with zero anchor.
\end{example}
General Lie algebroids are, loosely, a mixture of the two above examples. The mantra here is that whatever can be done with the tangent bundle can be done in the setting of Lie algebroids. In particular, we have the notion of torsion. The \emph{torsion tensor} of a linear connection on a Lie algebroid $(A, \rho , [-,-])$ is given by
\begin{equation}
T_\nabla(u,v) := \nabla_u v - (-1)^{\widetilde{u} \widetilde{v}}\, \nabla_v u - [u,v]\,,
\end{equation}
for all $u,v \in \InSec(A)$.  
\begin{proposition}\label{prop:Torsion}
Let $\nabla^{(1)}, \nabla^{(2)}$ and $\nabla^{(3)}\in \mathcal{C}(A)$ be  connections on a Lie algebroid $(A, \rho, [-,-])$. Then
$$T_{[\nabla^{(1)}, \nabla^{(2)}, \nabla^{(3)}]} = T_{\nabla^{(1)}} - T_{\nabla^{(2)}} + T_{\nabla^{(3)}}\,.$$
\end{proposition}
\begin{proof}
Directly, given any $u,v \in \InSec(A)$,
\begin{align*}
T_{[\nabla^{(1)}, \nabla^{(2)}, \nabla^{(3)}]}(u,v)&= [\nabla^{(1)}, \nabla^{(2)}, \nabla^{(3)}]_u v - (-1)^{\widetilde{u}  \widetilde{v}}\,[\nabla^{(1)}, \nabla^{(2)}, \nabla^{(3)} ]_v u - [u,v]\\
& = \nabla^{(1)}_u v - (-1)^{\widetilde{u}  \widetilde{v}}\,\nabla^{(1)}_v u - [u,v]\\
& -\nabla^{(2)}_u v + (-1)^{\widetilde{u}  \widetilde{v}}\,\nabla^{(2)}_v u + [u,v]\\
& + \nabla^{(3)}_u v - (-1)^{\widetilde{u}  \widetilde{v}}\,\nabla^{(3)}_u v - [u,v] \\
& = T_{\nabla^{(1)}}(u,v) - T_{\nabla^{(2)}}(u,v) + T_{\nabla^{(3)}}(u,v)\,.
\end{align*}
\end{proof}
\begin{remark}
The above proposition shows that torsion can be considered as a heap homomorphism from the heap of connections to the heap of vector-valued two forms.
\end{remark}
\begin{corollary}\
\begin{enumerate}
\setlength\itemsep{0.5em}
\item The subset of torsion-free connections $\cC_{TF}(A) \subset \cC(A)$ forms an abelian subheap of $(\cC(A), [-,-,-])$.
\item $T_{[\nabla^{(1)}, \nabla^{(2)}, \nabla^{(3)}]} + T_{[\nabla^{(3)}, \nabla^{(1)}, \nabla^{(2)}]} + T_{[\nabla^{(2)}, \nabla^{(3)}, \nabla^{(1)}]} = T_{\nabla^{(1)}} + T_{\nabla^{(2)}} + T_{\nabla^{(3)}}$.
\end{enumerate} 
\end{corollary}
\begin{example} Continuing Example \ref{exp:RieMan}, the torsion-free connections on a Riemannian manifold $(M,g)$, form an abelian subheap of the abelian heap of all affine connections. Moreover, the set of torsion-free connections forms canonically comes with an abelian group structure. 
\end{example}
Following Brzeziński, from Definition 2.9 and Proposition 2.10 of \cite{Brzezinski:2020}, we know that there is a subheap relation $\sim_{\cC_{TF}(A)}$ on $\cC(A)$ defined by $\nabla^{(1)} \sim_{\cC_{TF}(A)}\nabla^{(2)}$ if there exists a $\nabla \in \cC_{TF}(A)$ such that $[\nabla^{(1)}, \nabla^{(2)}, \nabla] \in \cC_{TF}(A) $. In fact, if two connections are equivalent, then $[\nabla^{(1)}, \nabla^{(2)}, \nabla']\in \cC_{TF}(A)$ for all $\nabla'\in \cC_{TF}(A)$. It is known that such a relation defines an equivalence relation. \par 
We then observe from Proposition \ref{prop:Torsion}, assuming we have equivalent connections, $T_{[\nabla^{(1)}, \nabla^{(2)}, \nabla]} = T_{\nabla^{(1)}} - T_{\nabla^{(2)}} = 0$. Thus, the torsion tensors for two equivalent connections must be equal. In other words, connections with the same torsion are representatives of the same equivalence class within the heap of linear connections.

The \emph{curvature tensor} of a connection on a Lie algebroid is given by
\begin{equation}
R_\nabla(u,v)w := \nabla_u \nabla_v w - (-1)^{\widetilde{u} \widetilde{v}}\, \nabla_v \nabla_u w - \nabla_{[u,v]}w\, ,
\end{equation}
for arbitrary $u,v$ and $w \in \InSec(A)$.
\begin{proposition}
Let $\nabla^{(1)}, \nabla^{(2)}$ and $\nabla^{(3)}\in \mathcal{C}(A)$ be connections on a Lie algebroid $(A, \rho, [-,-])$. Then
\begin{align*}
R_{[\nabla^{(1)}, \nabla^{(2)}, \nabla^{(3)}]}(u,v)w &= R_{\nabla^{(1)}}(u,v)w + R_{\nabla^{(2)}}(u,v)w + R_{\nabla^{(3)}}(u,v)w\\
&+ 2 \nabla^{(2)}_{[u,v]}w + \sum_{i \neq j}(-1)^{ij}[\nabla^{(i)}_u, \nabla^{(j)}_v]w\,,
\end{align*}
for arbitrary $u,v$ and $w \in \InSec(A)$.
\end{proposition}
\begin{proof}
Via direct calculation 
\begin{align*}
R_{[\nabla^{(1)}, \nabla^{(2)}, \nabla^{(3)}]}(u,v)w &= [\nabla^{(1)}_u , \nabla^{(2)}_u, \nabla^{(3)}_u][\nabla^{(1)}_v , \nabla^{(2)}_v, \nabla^{(3)}_v]w\\
& - (-1)^{\widetilde{u} \widetilde{v}}\,  [\nabla^{(1)}_v , \nabla^{(2)}_v, \nabla^{(3)}_v][\nabla^{(1)}_u , \nabla^{(2)}_u, \nabla^{(3)}_u]w
- \big[\nabla^{(1)}_{[u,v]}, \nabla^{(2)}_{[u,v]}, \nabla^{(3)}_{[u,v]}\big]w\\
&= \big[\nabla^{(1)}_u , \nabla^{(1)}_v \big]w - \nabla^{(1)}_{[u,v]}w
+\big[\nabla^{(2)}_u , \nabla^{(2)}_v \big]w + \nabla^{(2)}_{[u,v]}w \nabla^{(2)}_{[u,v]}w-\nabla^{(2)}_{[u,v]}w \\
&+ \big[\nabla^{(3)}_u , \nabla^{(3)}_v \big]w - \nabla^{(3)}_{[u,v]}w
+\sum_{i \neq j}(-1)^{ij}[\nabla^{(i)}_u, \nabla^{(j)}_v]w\\
&= R_{\nabla^{(1)}}(u,v)w + R_{\nabla^{(2)}}(u,v)w + R_{\nabla^{(3)}}(u,v)w\\
&+ 2 \nabla^{(2)}_{[u,v]}w + \sum_{i \neq j}(-1)^{ij}[\nabla^{(i)}_u, \nabla^{(j)}_v]w\,.
\end{align*}
\end{proof}
Setting $\nabla = \nabla^{(1)} = \nabla^{(2)} = \nabla^{(3)}$ allows for a quick check: 
$$R_\nabla(u,v)w = 3 R\nabla(u,v)w + 2 \nabla_{[u,v]}w - 2 [\nabla_u , \nabla_v]w=  3 R_\nabla(u,v)w - 2  R_\nabla(u,v)w=  R_\nabla(u,v)w\,,$$
and so the proposition is consistent. 

\begin{corollary}
Let $\nabla^{(i)}\in \mathcal{C}(A)$ be flat connections, then
$$R_{[\nabla^{(1)}, \nabla^{(2)}, \nabla^{(3)}]}(u,v)w = 2 \nabla^{(2)}_{[u,v]}w + \sum_{i \neq j}(-1)^{ij}[\nabla^{(i)}_u, \nabla^{(j)}_v]w\,.$$
Specifically, the connection $[\nabla^{(1)}, \nabla^{(2)}, \nabla^{(3)}]$ need not be flat.
\end{corollary}
The above corollary tells us that, in, general, the heap structure on connections on a Lie algebroid does \emph{not} close on the subset of flat connections. 

\subsection{The Endomorphism Truss of Connections}
\begin{definition}
Let $(A, \rho)$   be an anchored vector bundle. Then an \emph{anchored vector bundle endomorphism} is a vector bundle map (over the identity) $\phi : A \rightarrow A$ that preserves the anchor, i.e., $\rho = \rho \circ \phi$. Furthermore, if $(A, \rho, [-,-])$ is a Lie algebroid, then an anchored vector bundle endomorphism is a \emph{Lie algebroid endomorphism} if $\phi[u,v] =  [\phi(u), \phi(v)]$ for all $u,v \in \InSec(A)$.
\end{definition}
For this subsection we need only discuss anchored vector bundles, there is no real change at all when extending to Lie algebroids. In particular, the Lie bracket plays no r\^ole in the following.
\begin{definition}
Let $(A,  \rho)$ be an anchored vector bundle. An \emph{endomorphism of the set of connections} $\Phi =(\phi , \omega): \cC(A) \rightarrow \cC(A)$ consists of two parts:
\begin{enumerate}
\item an anchored vector bundle endomorphism $\phi : A \rightarrow A$,
\item an even $(1,2)$-tensor, thought of as a $C^\infty(M)$-linear map $\omega : \InSec(A)\times \InSec(A) \rightarrow \InSec(A)$,
\end{enumerate}
and defined as
$$(\Phi \nabla)_u v :=  \nabla_{\phi(u)}v + \omega(u,v)\,,$$
for arbitrary $u,v\in \InSec(A)$. Composition is defined in an obvious way, i.e., if $\Phi = (\phi , \omega)$ and $\Phi' = (\phi' , \omega')$, then  $\Phi \circ \Phi' := (\phi \circ \phi', \omega + \omega')$. The monoid of endomorphisms of connections we denote as $\End(\cC(A))$.
\end{definition}
For completeness, we explicitly check that $\Phi\nabla$ is indeed a connection on $(A, \rho)$.  First, $\Phi\nabla$ is (Grassmann) even. Second,
$$(\Phi\nabla)_{fu}v = \nabla_{\phi(fu)}v + \omega(fu,v) = \nabla_{f \phi( u)}v + f \,\omega(u,v) = f\big( \nabla_{\phi(u)}v + \omega(u,v)\big)\,. $$
Thirdly, 
\begin{align*}
 (\Phi\nabla)_{u}fv &=\nabla_{\phi(u)}fv + \omega(u,fv) = \rho_{\phi(u)}(f) \,v + (-1)^{\widetilde{u}\widetilde{f}} f\, \big(\nabla_{\phi(u)}v + \omega(u,v) \big)\\ 
 &= \rho_{u}(f)\,v + (-1)^{\widetilde{u}\widetilde{f}}f\, \big(\nabla_{\phi(u)}v + \omega(u,v) \big) \,.
\end{align*}
Thus, comparing with Definition \ref{Def:LinCon}, we see that $\Phi\nabla$ is a linear connection. \par
\begin{proposition}
Let $(A, \rho)$ be an anchored vector bundle. Then $\End(\cC(A))$ is a truss. 
\end{proposition}
\begin{proof}
This follows from general facts about endomorphisms of abelian heaps, see \cite[Section 3.7]{Brzezinski:2020}. We explicitly construct $[\Phi^{(1)}, \Phi^{(2)}, \Phi^{(3)}]$  via its evaluation. Specifically,
\begin{align*}
\big([\Phi^{(1)}, \Phi^{(2)}, \Phi^{(3)}]\nabla\big)_u v & := [\Phi^{(1)}\nabla, \Phi^{(2)}\nabla, \Phi^{(3)}\nabla]_u v \\
&= \nabla_{\phi^{(1)}u}v - \nabla_{\phi^{(2)}u}v +\nabla_{\phi^{(3)}u}v\\
&+ \omega^{(1)}(u,v)- \omega^{(2)}(u,v) + \omega^{(3)}(u,v)\,.
\end{align*}
Thus, we have the structure of an abelian heap on $\End(\cC(A))$. The binary product is given as 
$$\big( \Phi \circ \Phi' \nabla \big)_u v := \nabla_{\phi \phi'u}v + \omega(u,v) + \omega'(u,v )\,.$$
Defining $\Phi \circ [\Phi^{(1)}, \Phi^{(2)}, \Phi^{(3)}] := [\Phi \circ\Phi^{(1)}, \Phi \circ\Phi^{(2)}, \Phi\circ\Phi^{(3)}]$, we have the left distributivity property. Explicitly,
\begin{align*}
\big(\Phi \circ[\Phi^{(1)}, \Phi^{(2)}, \Phi^{(3)}]\nabla\big)_u v & := [\Phi \Phi^{(1)}\nabla, \Phi\Phi^{(2)}\nabla, \Phi \Phi^{(3)}\nabla]_u v\\
&= \nabla_{\phi \phi^{(1)}u}v - \nabla_{\phi \phi^{(2)}u}v +\nabla_{\phi\phi^{(3)}u}v\\
&+\big( \omega^{(1)}(u,v) + \omega(u,v)\big)- \big(\omega^{(2)}(u,v) + \omega(u,v)\big)\\ 
& + \big(\omega^{(3)}(u,v) + \omega(u,v)\big)\,.
\end{align*}
The right distributivity property follows similarly. 
\end{proof}
\begin{example}\label{exp:ShiftAffConn}
Consider the tangent bundle $\sT M$ of a supermanifold. As a Lie algebroid, the anchor is the trivial map and the Lie bracket is the standard Lie bracket of vector fields. The set of affine connections we denote as $\cC(M)$. We fix $\phi : \sT M \rightarrow \sT M$ to be the identity map. Then any even $(1,2)$-tensor defines a ``restricted'' endomorphism of $\cC(M)$ or a ``general shift'' via
$$\nabla_X Y \longmapsto \nabla_X Y + \omega(X,Y)\,.$$
The truss structure is defined as 
$$[\omega^{(1)}, \omega^{(2)}, \omega^{(3)}] := \omega^{(1)} -\omega^{(2)} +\omega^{(3)}\,, $$
and
the binary product being addition of tensors, i.e., $\omega \circ \omega' = \omega+ \omega'$. The distributivity of the binary product over the ternary product is evident, i.e.,
$$\omega \circ [\omega^{(1)}, \omega^{(2)}, \omega^{(3)}] := \big( \omega^{(1)} + \omega\big)- \big(\omega^{(2)} + \omega\big) + \big(\omega^{(3)} + \omega\big)\,,$$
and similarly for the right distributivity. 
\end{example}
\begin{remark}
General shifts in affine connections, see Example \ref{exp:ShiftAffConn}, are well-known in the literature and are studied in the context of metric-affine gravity (MAG) (see for example, \cite{Iosifidis:2020} and references therein). The new aspect here is the realisation that there is a truss behind these shifts.
\end{remark}

\section{Concluding Remarks}
In this note, we have presented the heap structure on the set of linear connections on a Lie algebroid.  Some preliminary consequences, such as the torsion and curvature of a triple product of connections, and we give the construction of the truss associated with endomorphisms of the set of linear connections. \par   
We remark that the results of this note extend verbatim to the algebraic setting of (left) connections on anchored modules and Lie--Rinehart pairs over associative supercommutative, unital superalgebras. Indeed, we have avoided using local descriptions in any calculations. With a little effort, we expect the results presented here to generalise to the setting of almost commutative Lie algebroids (see \cite{Ngakeu:2017}).

\section*{Acknowledgements} 
The author thanks Tomasz Brzeziński for comments on earlier drafts of this note.

%

%

\end{document}